\documentclass[12pt]{amsart}

\usepackage{amssymb}

\usepackage{tikz} 

\AtBeginDocument{
\newenvironment{enumeratea}
{\bgroup

\begin{enumerate}}
{\end{enumerate}\egroup}

\newenvironment{enumeratei}
{\bgroup

\begin{enumerate}}
{\end{enumerate}\egroup}
}

\def\E6#1#2#3#4#5#6{\begin{matrix} 
#1 & #2 & #3 & #4 & #5
\\
 & & #6 & & 
 \end{matrix}}
 
\def\EE7#1#2#3#4#5#6#7{\begin{smallmatrix} 
#1 & #2 & #3 & #4 & #5 & #6
\\
 & & #7 & & & 
 \end{smallmatrix}}
 
\def\EEE8#1#2#3#4#5#6#7#8{\begin{smallmatrix} 
#1 & #2 & #3 & #4 & #5 & #6 & #7 
\\
 & & #8 & & & & 
 \end{smallmatrix}}

\numberwithin{equation}{section}

\newtheorem{thm}{Theorem}[section]
\newtheorem{prop}[thm]{Proposition} 
\newtheorem{lem}[thm]{Lemma}
\newtheorem{cor}[thm]{Corollary}

\theoremstyle{definition}
\newtheorem{nota}[thm]{Notation}

\def\eg{{e.g.}\kern.3em}
\def\ie{{i.e.}\kern.3em}
\def\resp{resp.\kern.3em}
\def\loccit{\textit{loc.\kern2pt cit.}\kern.3em}

\def\Cor{Cor.\kern.2em} 
\def\Lem{Lem.\kern.2em} 
\def\Prop{Prop.\kern.2em} 
\def\Th{Th.\kern.2em} 

\def\calD{\mathcal{D}}

\def\NN{\mathbb{N}}
\def\ZZ{\mathbb{Z}}
\def\CC{\mathbb{C}}

\def\delodd#1{\delta_{\text{odd},#1}}

\let\vep\varepsilon

\begin{document}

\title{Classification of the root systems $R(m)$} 

\author{Patrick Polo} 

\address{Institut de Mathématiques de Jussieu-Paris Rive Gauche\\
Sorbonne Université -- Campus Pierre et Marie Curie\\
4, place Jussieu -- Boîte Courrier 247\\
F-75252 Paris Cedex 05\\
France}

\email{patrick.polo@imj-prg.fr}



\begin{abstract} 
  Let $R$ be a reduced irreducible root system, $h$ its Coxeter number and $m$
a positive integer smaller than $h$. Choose a base of $R$, whence a
corresponding height function, and let $R(m)$ be the set of roots whose height
is a multiple of $m$. In a recent paper, S. Nadimpalli, S. Pattanayak and D.
Prasad studied, for the purposes of character theory at torsion elements, the
root systems $R(m)$; in particular, they
introduced a constant $d_m$ which is always the dimension of a representation
of the semisimple, simply-connected group with root system dual to $R(m)$ and equals $1$
if the roots of height $m$ form a base of $R(m)$, and proved this property when
$R$ is of type $A$ or $C$, and also in type $B$ if $m$ is odd. In this paper,
we complete their analysis by determining a
base of $R(m)$ and computing the constant $d_m$ in all cases.
\end{abstract} 

\date{May 14, 2025}  
\maketitle 

\section{Introduction} 

Let $R$ be a reduced, irreducible root system in a real vector space $V$. 
Let $h$ be its Coxeter number and $m$ an integer such that $2\leq m < h$. 
Choose a set $R^+$ of positive roots, whence a base $\Delta$ of $R$ and the corresponding 
height function $\mathrm{ht}$ on $R$, 
and let $R(m)$ denote the set of roots whose height is a multiple of $m$. 
It is a closed root subsystem, with set of positive roots $R^+(m) = R(m)\cap R^+$. 
Let $\Gamma(m)$ denote the set of simple roots of $R(m)$ contained in $R^+(m)$. 
Denote by $R_m$ the set of roots of height $m$. One has  $\beta-\gamma \not\in R$ 
for every $\beta\not= \gamma$ in $R_m$, and therefore $R_m \subset \Gamma(m)$.

The root systems $R(m)$ were studied in \cite{NPP25} in the following context. Let 
$X \subset V^*$ be the lattice dual to $\ZZ R$, let $\Phi$, \resp $\Phi_m$, be the root system dual to 
$R$, \resp $R(m)$, 
and let $G$ and $H(m)$ be split reductive groups over $\CC$ with root data $(X,\ZZ R, \Phi, R)$ and 
$(X,\ZZ R, \Phi_m, R(m))$ respectively.  
In \cite{NPP25}, the authors consider a certain reductive group $G(m)$ isogenous to $H(m)$ 
and, in their study of the values of irreducible characters $\Theta_\lambda$ of $G$ on a conjugacy class of principal elements of order $m$, they introduced the constant 
\begin{equation}\label{def-d_m}
d_m = \prod_{\beta\in R^+(m)} \dfrac{ \langle \rho/m, \beta \rangle }{  \langle \rho_m, \beta \rangle} , 
\end{equation} 
where $\rho$, \resp $\rho_m$, is the half-sum of the elements of $\Phi^+$, \resp $\Phi_m^+$, 
and observed that, by Weyl dimension formula, $d_m$ is the dimension of the irreducible representation of 
$H(m)$ whose highest weight $\omega$ corresponds to the nodes of the Dynkin diagram of $\Gamma(m)$ which are not in $R_m$, see \loccit  \Th 4.1. 
Further, in their Section 5 they proved that  $R_m = \Gamma(m)$, whence $d_m = 1$, 
when $R$ is of type $A_n$ or $C_n$ and also in type $B_n$ if $m$ is odd, 
and in Sections 13--17 of their paper they described their group $G(m)$ when $m \vert h$. 

We complete their results by determining in all cases $\Gamma(m)$, the weight $\omega$ (which, if non-zero, is always a minuscule fundamental weight) and the constant $d_m$.

{\sc Acknowledgements.} This paper was written during a stay of the author at IIT Bombay in March-April 2025. 
I warmly thank Dipendra Prasad for the invitation to IITB and for asking me the question of 
classifying the root subsystems $R(m)$ beyond the results  of \cite{NPP25} and computing their constant $d_m$.

\section{Statement of the result} 

We use the following notation. In type $D_{n+1}$ we number the vertices of the Dynkin diagram by $0,1,\dots, n$, so that $D_2 = 2A_1$ corresponds to the vertices $0,1$. That is, the simple roots are: 
\begin{equation}\label{simple-D}
\alpha_0 = \vep_1 + \vep_0 \qquad \text{and} \qquad \alpha_i = \vep_i - \vep_{i-1} \quad \text{for } i = 1,\dots, n.
\end{equation} 
Similarly in type $B_n$  the simple roots are: 
\begin{equation}\label{simple-B}
\alpha_1 = \vep_1   \qquad \text{and} \qquad \alpha_i = \vep_i - \vep_{i-1} \quad \text{for } i = 2,\dots, n.
\end{equation} 
For $k\in \mathbb{N}^*$ such that $3k\leq n$, consider the following roots of height $4k$ in $D_{n+1}$ and $B_n$ respectively: 
\begin{equation}\label{def-gamma_k}
\begin{cases}  
\gamma_{4k} = \alpha_0 + \alpha_1 + 2 (\alpha_2 + \cdots + \alpha_k) + \alpha_{k+1} + \cdots + \alpha_{3k}  , 
\\
\gamma'_{4k}   = 2 (\alpha_1 + \cdots + \alpha_k) + \alpha_{k+1} + \cdots + \alpha_{3k}.  
\end{cases} 
\end{equation} 
Note that in both cases this root is $\vep_k + \vep_{3k}$. 
Further, 
with the obvious inclusions $D_4 \subset E_6$ and $D_7 \subset E_8$, we may consider $\gamma_4$ as a root of height $4$ in $E_6, E_7, E_8$ and $\gamma_8$ as a root of height $8$ in $E_8$.

\smallskip 
Let us say that a root subsytem of $R$ is of \emph{Levi type} if it  is $W$-conjugate (where $W$ is the Weyl group) to the root subsystem generated by a subset $I$ of $\Delta$.

\begin{thm}\label{main-thm} {\rm (1)} $R_m$ is always contained in a $W$-conjugate of $\Delta$. 

\smallskip {\rm (2)} One has $\Gamma(m) = R_m$ and hence $R(m)$ is of Levi type, in all cases except the cases below. 
In these cases, $\Gamma(m)$ is the union of $R_m$ and a root $\delta_{2m}$ of height $2m$,  is not of Levi 
type and, denoting by $X^\dag$ the connected component of $\Gamma(m)$ containing $\delta_{2m}$ we list the types of 
$X^\dag$ and $X^0 = X^\dag - \{ \delta_{2m} \}$ and the dimension $d_m$ of the fundamental representation of $G(m)$ corresponding to the simple coroot $\delta_{2m}$:  
\begin{enumeratea} 
\item $m = 2k$ and $R$ of type $D_{n+1}$ or $B_n$ with $n\geq 3k$. Then $\delta_{2m} = \vep_k + \vep_{3k}$ and, setting 
$q = \lfloor (n-k)/m \rfloor$ one has $X^\dag = D_{q+1}$ and $X^0 = A_q$, whence $d_m = 2^q$. 

\smallskip 
\item $m = 2$ and $R$ of type $E_6, E_7, E_8$. Then  $\delta_{4} = \gamma_4$ and one has the table: 
$$
\begin{array}{|c|c|c|c|}
\hline 
R & X^\dag & X^0 & d_2 
\\
\hline
E_6 & A_1 & \varnothing & 2 
\\
\hline
E_7 & A_7 & A_6 & 8 
\\
\hline
E_8 & D_8 & D_7 & 16 
\\
\hline
\end{array}
$$

\smallskip 
\item $m = 3$ and $R$ of type $E_6, E_7, E_8$.  Then $\delta_6$ is the sum of the simple roots in type $E_6$ and one has the table: 
$$
\begin{array}{|c|c|c|c|}
\hline 
R & X^\dag & X^0 & d_3 
\\
\hline
E_6 & A_2 & A_1 & 3
\\
\hline
E_7 & A_2 & A_1 & 3
\\
\hline
E_8 & A_8 & A_7 & 9
\\
\hline
\end{array}
$$

\smallskip 
\item $m \in \{4,5,8\}$ and $R$ of type $E_8$. Then one has the table: 
$$
\begin{array}{|c|c|c|c|c|}
\hline 
m & \delta_{2m}  & X^\dag & X^0 & d_m
\\
\hline 
4 & \gamma_8 & A_3 & A_2 & 4 
\\
\hline 
5 & \EEE8{1}{2}{2}{1}{1}{1}{1}{1} & A_4 & A_3 & 5 \rule{0pt}{12pt} 
\\[2pt]
\hline 
8 &  \EEE8{1}{2}{3}{3}{3}{2}{1}{1} & A_1 & \varnothing & 2  \rule{0pt}{12pt} 
\\[2pt] 
\hline 
\end{array}
$$
\smallskip 
\item $m \in \{2,3\}$ and $R$ of type $F_4$. Denoting the simple roots by $a,b,c,d$, with $a,b$ long and $b+c \in R$, 
one has the table: 
$$
\begin{array}{|c|c|c|c|c|}
\hline 
m &  \delta_{2m} & X^\dag & X^0 & d_m 
\\
\hline 
2 & a + b + 2c & A_1 & \varnothing & 2 
\\
\hline 
3 & a + b + 2c + 2d & A_2 & A_1 & 3 
\\
\hline 
\end{array}
$$

\smallskip 
\item $m = 2$ and $R$ of type $G_2$. Let $\alpha, \beta$ be the simple roots,  with $\beta$ long. Then  
$R^+(2) = \{\beta+\alpha, \beta+3\alpha\}$ is of type  $A_1 + A_1^\dag$, hence $d_2  = 2$. 
\end{enumeratea} 

\smallskip {\rm (3)} So, when $d_m\not=1$ it corresponds to the natural representation or a spin one of a simple factor of $G(m)$ of type $A$ or $D$; in particular this is 
a minuscule representation. 
\end{thm} 

\section{Proofs}  

\subsection{The partition of heights}\label{dual-part}  
For $k\in \mathbb{N}^*$, denote by $\pi_k$ the number of elements of $R^+$ of height $k$. 
Recall that the sequence $(\pi_1,\pi_2,\dots)$ is the dual of the partition of exponents of $R$; see 
\cite{Kos59}, \Cor 8.7 or  \cite{Mac72}, \Cor 2.7. 
Thus, one obtains in types $B_n$,  $C_n$ and $D_{n+1}$ that:
\begin{equation}\label{card-R_m} 
\pi_k =  
\begin{cases} 
n+1  - \lfloor k/2\rfloor & \text{ if $R$ has type $D_{n+1}$ and } 1\leq k \leq n; 
\\
n  - \lfloor k/2\rfloor & \text{ else,} 
\end{cases}
\end{equation} 
while in type $A_{n-1}$ one has $\pi_k = n-k$ for $k=1,\dots, n-1$.

\subsection{Notation} In types $B_n, C_n$ and $D_{n+1}$, for $i = 0,\dots, m$, let $q(i)$ denote the largest integer such that $i + m q(i) \leq n$, \ie 
\begin{equation}\label{def-q-C}
q(i) =  \lfloor \dfrac{n-i}{m} \rfloor .  
\end{equation}

\subsection{Types $A_{n-1}$ and $C_n$} In these cases, it was already proved in \cite{NPP25}, \Prop 5.1 that 
$\Gamma(m) = R_m$. 
In type $A$ it is well-know that every closed  root subsystem is of Levi type, whence the result in this case. 

For the sake of completeness, let us describe $R_m$ in type $A_{n-1}$ when $m < n$. 
In view of our numbering in cases $B,D$, we take as simple roots the roots $\vep_{i+1} - \vep_i$ for $i = 1,\dots, n-1$. 
Then $R_m$ consists of the roots $\vep_{i+m} - \vep_i$, for all $i\geq 1$ such that $i+m\leq n$. If $2m\leq n$, these roots form 
a union of $m$ Dynkin diagrams of type $A$, parametrised by the residue classes modulo $m$. 
Namely, if we write $n = qm + t$, with $q\geq 2$ and $0\leq t < m$, then the roots 
$\vep_{r+ im} - \vep_{r+ (i-1)m}$, for $r = 1,\dots, t$ and $i=1,\dots, q$, form a Dynkin diagram of type $A_q$, while 
for $r = t+1,\dots, m$ and $i=1,\dots, q-1$ they form a Dynkin diagram of type $A_{q-1}$. If $m < n < 2m$, \ie if $n = m + t$ with 
$1\leq t < m$, then $R_m$ consists only of the roots $\vep_{i+m} - \vep_i$, for $i=1,\dots, t = n{-}m$ and has type 
$(n{-}m) A_1$. Thus, one obtains the 

\begin{prop} Suppose that $R$ is of type $A_{n-1}$, with $n\geq 3$. 
If $n< 2m$ then $R(m)$ has type 
$(n{-}m) A_1$. Else, if $n = qm+t$ with $q\geq 2$ and $0\leq t < m$, then $R(m)$ has type 
$t A_q + (m-t) A_{q-1}$. 
\end{prop} 

\smallskip Next, in type $C_n$ with $n\geq 2$, we number the simple roots as follows:  
\begin{equation}\label{simple-C}
\alpha_1 = 2\vep_1   \qquad \text{and} \qquad \alpha_i = \vep_i - \vep_{i-1} \quad \text{for } i = 2,\dots, n.
\end{equation} 
Then for $j \geq i$ one has $\mathrm{ht}(\vep_j - \vep_i) = j-i$ and $\mathrm{ht}(\vep_j + \vep_i) = j+i-1$; in particular 
$\mathrm{ht}(2\vep_i) =2i-1$. For $r = 1,\dots, m$, set $\bar{r} = m+1-r$. 

\smallskip 
Set $k = \lfloor m/2\rfloor$, so that $m = 2k$ or $m= 2k+1$. For the sake of exposition, assume first that $m\leq n$. 
Then $R_m$ consists of the roots:
\begin{equation} 
\begin{array}{l} 
 \alpha_{i,b} = \vep_{i+bm} - \vep_{i+(b-1) m}, \; \text{for $i = 1,\dots, m$ and } b = 1,\dots, q(i); 
 \\[3pt] 
\beta_{r} = \vep_r + \vep_{\bar{r}}, \qquad \qquad \quad \, \text{for } r = 1,\dots, \lfloor (m+1)/2\rfloor. 
\end{array}
\end{equation} 
For $r = 1,\dots, k$, one obtains the following Dynkin diagram $\mathcal{D}(r)$:  
\begin{center} 
\begin{tikzpicture}
\coordinate (1) at (0,1) node at (1) {\small $\alpha_{r,q(r)}$ };
\draw (.55,1) -- (2.1,1);
\coordinate (3) at (2.5,1) node at (3) {\small \dots };
\draw (2.9,1) -- (4.2,1);
\coordinate (5) at (5.35,1) node at (5) {\small $\vep_{2m+r} {-} \vep_{m+r}$};
\draw (6.45,1) -- (7.65,1);
\coordinate (7) at (8.45,1) node at (7) {\small $\vep_{m+r} {-} \vep_r$};
\draw (8.45,.8) -- (4.6,.1);
\coordinate (0) at (4.1,0) node at (0) {\small $\vep_{r} {+} \vep_{\bar{r}}$};
\draw (3.7,-.15) -- (0,-.8);
\coordinate (2) at (0,-1) node at (2) {\small $\vep_{ m+\bar{r} } {-} \vep_{ \bar{r} }$};
\draw (.8,-1) -- (1.9,-1);
\coordinate (4) at (3,-1) node at (4) {\small $\vep_{ 2m+\bar{r} } {-} \vep_{ m+\bar{r} }$};
\draw (4.1,-1) -- (5.6,-1);
\coordinate (6) at (6,-1) node at (6) {\small \dots};
\draw (6.4,-1) -- (7.9,-1);
\coordinate (8) at (8.45,-1) node at (8) {\small $\alpha_{\bar{r} ,q(\bar{r})}$};
\end{tikzpicture} 
\end{center} 
which is of type $A_{d(r)}$ with $d(r) = q(r)+1 + q(\bar{r})$. 

Further, for $r = 1,\dots, k-1$, set $\eta_r = -\vep_{\bar{r} + q(\bar{r}) m} - \vep_{r+1 + q(r+1)m}$. Then $\eta_r$ is orthogonal 
to all roots of $R_m$ except that it is linked by a simple bond to the right end-point of $\mathcal{D}(r)$, which is 
$\alpha_{\bar{r} ,q(\bar{r})}$ if $q(\bar{r})\geq 1$ and $\beta_r$ else, and to the left end-point of  
$\mathcal{D}(r+1)$, which is $\alpha_{r+1,q(r+1)}$  if $q(r+1)\geq 1$ and $\beta_{r+1}$ else.  
Moreover, the $\eta_r$'s are pairwise orthogonal. Thus, we obtain the Dynkin diagram: 
\begin{center} 
\begin{tikzpicture}
\coordinate (1) at (0,1) node at (1) {\small $A_{1,d(1)}$ };
\draw (.55,1) -- (1.45,1);
\coordinate (2) at (1.65,1) node at (2) {\small $\eta_1$ };
\draw (1.85,1) -- (2.7,1);
\coordinate (3) at (3.25,1) node at (3) {\small $A_{2,d(2)}$ };
\draw (3.8,1) -- (4.7,1);
\coordinate (4) at (5.05,1) node at (4) {\small \dots};
\draw (5.3,1) -- (6.1,1);
\coordinate (5) at (6.5,1) node at (5) {\small $\eta_{k-1}$ };
\draw (6.9,1) -- (7.7,1);
\coordinate (6) at (8.3,1) node at (6) {\small $A_{k,d(k)}$};
\end{tikzpicture} 
\end{center} 
which is of type $A_N$, with $N = k-1 + \sum_{r=1}^k (1+q(r) + q(\bar{r}))$.  

\medskip If $m = 2k$, the diagrams $\mathcal{D}(1),\dots, \mathcal{D}(k)$ exhaust $R_m$. Since we know that 
$\vert R_m\vert = n - k$ 
we see that $N = n-1$. Then the long root $-2\vep_{1+q(1)m}$ is linked by a double bond to the 
left end-point of $\calD(1)$, which is $\alpha_{1,q(1)}$ if $q(1) \geq 1$ and $\vep_1 + \vep_m$ else, 
and is orthogonal to all other roots in the previous diagram of type $A_{n-1}$, so we obtain a diagram of type $C_n$.

\smallskip Now, if $m = 2k+1$ then for $r = k+1$ one has $r = \bar{r}$ and the remaining roots in $R_m$ form a diagram 
$\mathcal{D}(k+1)$: 
\begin{center} 
\begin{tikzpicture}
\coordinate (0) at (0,1) node at (0) {\small $2\vep_{k+1}$};
\draw (.5,1.05) -- (1.8,1.05);
\draw (.5,.95) -- (1.8,.95);
\coordinate (2) at (1.15,1) node at (2) {$>$}; 
\coordinate (1) at (3,1) node at (1) {\small $\vep_{m+k+1} {-} \vep_{k+1}$};
\draw (4.15,1) -- (5.6,1);
\coordinate (5) at (6,1) node at (5) {\small \dots};
\draw (6.4,1) -- (7.9,1);
\coordinate (7) at (8.85,1) node at (7) {\small $\alpha_{k+1,q(k+1)}$};
\end{tikzpicture} 
\end{center} 
which is of type $C_{1 + q_{k+1}}$. (If $q(k+1)=0$ there is only the long root $2\vep_{k+1}$.) 
Set $\eta_0 = -\vep_{k+1 + q(k+1) m} - \vep_{1 + q(1)m}$. Then $\eta_0$ is linked to the left end-point 
of $\calD(1)$ by a single bond and to the right end-point of $\calD(k+1)$ by a simple (\resp double) bond if $q(k+1)$ is 
$\geq 1$ (\resp $=0$), and is orthogonal to all other roots of $R_m$ and to $\eta_1, \dots, \eta_{k-1}$. 
Again, since we know that $\vert R_m \vert + k = n$ we obtain a diagram of type $C_n$. 

Thus, we have proved in both cases that $R_m$ is of Levi type. It remains to determine the $q(i)$'s. Before doing this, 
let us indicate how the foregoing has to be modified when $n <m$, say $n = m -i$ with $i=1,\dots, \lfloor m/2 \rfloor$. 
Then of course $q(i) = 0$ for all $i$. Further, one has:  
$$
\calD(r) = 
\begin{cases} 
\varnothing & \text{ for } r = 1,\dots, i ; 
\\
\{\vep_r + \vep_{m+1-r}\} &  \text{ for } r = i+1,\dots, k; 
\\
\{2 \vep_{k+1}\} & \text{ if } m = 2k+1 
\end{cases}
$$ 
and the diagrams $\calD(i+1), \dots, \calD(k)$, together with the roots $\eta_i, \dots, \eta_{k-1}$ as before, form the following diagram  
\begin{center} 
\begin{tikzpicture}
\coordinate (0) at (-2,1) node at (0) {\small $\calD$ :};
\coordinate (1) at (0,1) node at (1) {\small $\vep_{i+1}{+}\vep_n$ };
\draw (.7,1) -- (1.4,1);
\coordinate (2) at (2.4,1) node at (2) {\small $-\vep_n{-}\vep_{i+2}$ };
\draw (3.25,1) -- (3.95,1);
\coordinate (3) at (4.9,1) node at (3) {\small $\vep_{i+2}{+}\vep_{n-1}$ };
\draw (5.8,1) -- (6.5,1);
\coordinate (4) at (6.9,1) node at (4) {\small \dots};
\draw (7.2,1) -- (7.9,1);
\coordinate (5) at (8.5,1) node at (5) {\small $\vep_{k}{+}\vep_{\bar{k}}$ };
\end{tikzpicture} 
\end{center} 
of type $A_{2k-2i-1}$. If $m=2k$, the diagrams $\calD(i+1), \dots, \calD(k)$ exhaust $R_m$ and $\calD$ together 
with the following diagram of type $C_{i+1}$:  
\begin{center} 
\begin{tikzpicture}
\coordinate (0) at (0,1) node at (0) {\small $-2\vep_1$};
\draw (.5,1.05) -- (1.8,1.05);
\draw (.5,.95) -- (1.8,.95);
\coordinate (2) at (1.15,1) node at (2) {$>$}; 
\coordinate (1) at (2.4,1) node at (1) {\small $\vep_1{-} \vep_2$};
\draw (3,1) -- (4.4,1);
\coordinate (5) at (4.8,1) node at (5) {\small \dots};
\draw (5.1,1) -- (6.6,1);
\coordinate (7) at (7.3,1) node at (7) {\small $\vep_i{-} \vep_{i+1}$};
\end{tikzpicture} 
\end{center} 
form a diagram of type $C_n$. If $m = 2k+1$, the only remaining root in $R_m$ is $2\vep_{k+1}$ and $\calD$ together 
with the following diagram of type $C_{i+2}$:  
\begin{center} 
\begin{tikzpicture}
\coordinate (0) at (0,1) node at (0) {\small $-2\vep_1$};
\draw (.5,1.05) -- (1.8,1.05);
\draw (.5,.95) -- (1.8,.95);
\coordinate (2) at (1.15,1) node at (2) {$>$}; 
\coordinate (1) at (2.7,1) node at (1) {\small $-\vep_{k+1} {-} \vep_1$};
\draw (3.6,1) -- (4.8,1);
\coordinate (5) at (5.4,1) node at (5) {\small $\vep_1{-} \vep_2$};
\draw (6,1) -- (7.3,1);
\coordinate (6) at (7.7,1) node at (6) {\small \dots};
\draw (8,1) -- (9.2,1);
\coordinate (7) at (9.9,1) node at (7) {\small $\vep_i{-} \vep_{i+1}$};
\end{tikzpicture} 
\end{center}  
form a diagram of type $C_n$. This shows that $R_m$ is also of Levi type when $n = m-i$, and it has type: 
\begin{equation}\label{Ccase n < m}
\begin{cases}
(k-i) A_1 & \text{ if } m=2k; 
\\
(k-i) A_1 + C_1 & \text{ if } m=2k+1. 
\end{cases}
\end{equation} 

Now, let us come back to the determination of the $q(i)$'s when $m\leq n$. 
One has the following, where $\delodd{m} = 1$ if $m$ is odd and $=0$ else.

\begin{lem}\label{lem-Cn} Write $n = qm+ t$, with $q\in \NN^*$ and $0\leq t < m$. 

\smallskip {\rm (a)} 
If $2t \leq m$ then $R_m$ has type $t A_{2q} + (k{-}t) A_{2q-1} + \delodd{m} \, C_q$. 

\smallskip {\rm (b)} 
If $2t\geq m{+}1$, then $R_m$ has type $(m{-}t) A_{2q} + (k{-m+}t) A_{2q+1} + \delodd{m} \,  C_{q+1}$. 
\end{lem}  

\begin{proof} Recall that for $r=1,\dots,\lfloor (m{+}1)/2\rfloor$, we have set  $\bar{r}= m{+1-}r$. 

\smallskip 
 (a) If $2t \leq m$ then all $q(\bar{r})$ equal $q-1$, whereas $q(r)= q$ for $r=1,\dots, t$ and $q(r) = q-1$ otherwise. Hence the union of the diagrams $\calD(1),\dots, \calD(r)$ has type $t A_{2q} + (k-t) A_{2q-1}$. Further, if $m=2k+1$ then 
$\calD(k+1)$ has type $C_q$. 

\smallskip (b) If $2t\geq m+1$, all $q(r)$ equal $q$, whereas $q(\bar{r})= q$ for $m+1-t \leq r \leq k$, and also for $r = k+1$ if $m = 2k+1$, 
and $q(\bar{r})= q-1$ if $r =1,\dots, m-t$. Hence the union of the diagrams $\calD(1),\dots, \calD(r)$ has type $(m-t) A_{2q} + 
(k-m+t) A_{2q+1}$; further, if $m=2k+1$ then $\calD(k+1)$ has type $C_{q+1}$. 
\end{proof} 

To relate conditions (a), (b) above to the notation used in \cite{NPP25}, note that 
(a) holds if and only if $2qm \leq 2n \leq (2q+1)m$. 
So, if $m$ divides $2n$ we must be in case (a); further, if $m$ is odd,  
one has necessarily $2n = 2qm$, and if $m$ is even,  
we will set $a = 2n/m$ as in \cite{NPP25}, \S 6. Then, we can summarize our findings in the following:

\begin{prop} Let $R$ be of type $C_n$, with $n\geq 2$. Then $R(m)$ is of Levi type and its type is given by 
\eqref{Ccase n < m} if $m>n$ and by Lemma {\rm \ref{lem-Cn}} if $m\geq n$. 
The cases where $m$ divides the Coxeter number $h = 2n$ are: 
\begin{enumeratea}
\item $m = 2k+1$ and $n = q m$. Then $R(m)$ has type $k A_{2q-1} + C_{q}$. 

\smallskip 
\item $m = 2k$. Then $R(m)$ has type $k A_{a-1}$, where $a=2n/m=n/k$. 
\end{enumeratea} 
\end{prop}

\subsection{Types $B_n$ and $D_{n+1}$} With notation as in \eqref{simple-D} and \eqref{simple-B}, 
for $i < j$ one has $\mathrm{ht}(\vep_j - \vep_i) = j-i$ and $\mathrm{ht}(\vep_j + \vep_i) = j+i$, 
and $\mathrm{ht}(\vep_i) = i$ in type $B_n$.  

Firstly, one has the following lemma. 

\begin{lem}\label{lem-BD} Let $R$ be of type $B_n$ or $D_{n+1}$, with $n\geq 3$. Then $\Gamma(m) = R_m$, except when 
$m = 2k$ and $3k\leq n$. In this case, there is the additional simple root $\vep_k + \vep_{3k}$ of height $2m$. 
\end{lem} 

\begin{proof} Assume that $\delta \in R^+(m)$ is not a sum of  
elements of $R_m$ and is minimal for this property. 
Then not all coefficients of $\delta$ are $1$, hence $\delta$ is of the form 
$\delta = \vep_k + \vep_{k+r}$ 
with $r\geq 1$ and $k+r \leq n$, and $\mathrm{ht}(\delta) = 2k+r$ is a multiple of $m$ stritctly greater than $m$. 
Thus, one has  $2(k+r) > 2k+r \geq 2m$  
and hence $k+r - m > 0$. 

So, if we substract the element $\vep_{k+r} - \vep_{k+r-m}$ of $R_m$ we are left with 
$\vep_k + \vep_{k+r-m}$, which is a root unless $k+r - m = k$, \ie $r = m$. 

Therefore, one has $r = m$, whence $k>0$, and $2k$ is a non-zero multiple $q m$ of $m$. Suppose that $q\geq 2$, 
then $k\geq m$ and hence, using the convention that $\vep_0 = 0$ if we are in type $B_n$ and $k=m$, one obtains that 
$$
\gamma = \delta  - (\vep_{k+m} - \vep_k) - (\vep_k - \vep_{k-m}) = \vep_k + \vep_{k-m} 
$$
is a root of height $(q-1)m$ and $\delta$ is the sum of $\gamma$ and two  
elements of $R_m$, contradicting the minimality of $\delta$. 
This contradiction shows that $q=1$, hence $2k = m$ and $\delta = \vep_k + \vep_{3k}$. 

Conversely, if $3k \leq n$ and $m = 2k$, the root $\delta_{4k}$ of $D_{n+1}$ (\resp $\delta'_{4k}$ of $B_{n}$) given by 
\eqref{def-gamma_k}
has height $2m$ and cannot be written as a sum of two roots of height $m$. This proves the lemma. 
\end{proof}

Set $k = \lfloor m/2\rfloor$, so that $m = 2k$ or $m=2k+1$. Again, for the sake of exposition, assume first that $m\leq n$. 

\begin{nota}\label{BD-def-bar} (a) In type $B_n$, we set $\bar{r} = m-r$ for $r = 1,\dots, m-1$. Then $m$ stands alone, \ie 
$\overline{m}$ is not defined. 

\smallskip 
(b) In type $D_{n+1}$, we set $\bar{r} = m-r$ for $r = 0,\dots, m$. 
\end{nota} 

Then $R_m$ consists of the roots:
 \begin{equation}\label{BD-R_m} 
 \begin{array}{l} 
 \alpha_{i,b} = \vep_{i+bm} - \vep_{i+(b-1) m} , \;  \text{for $1\leq  i < m$ and  } b = 1,\dots, q(i); 
 \\[3pt] 
 \beta_{r} = \vep_r + \vep_{\bar{r}}, \qquad \qquad \quad \,  \text{for $r$ such that } r < \bar{r} ; 
  \\[3pt] 
 \begin{cases}
  \vep_m \text{ and $\alpha_{m,b}$, for } b=1,\dots, q(m), & \text{ in type } B_n;  
  \\
  \alpha_{0,b} = \vep_{bm} - \vep_{(b-1) m} , \text{ for } b = 1,\dots, q(0), & \text{ in type } D_{n+1};  
 \end{cases} 
 \end{array} 
 \end{equation} 
For each $r = 1,\dots, k$ such that $r\not= \bar{r}$, one obtains the following Dynkin diagram $\calD(r)$: 
\begin{center} 
\begin{tikzpicture}
\coordinate (1) at (0,1) node at (1) {\small $\alpha_{r,q(r)}$ };
\draw (.55,1) -- (2.1,1);
\coordinate (3) at (2.5,1) node at (3) {\small \dots };
\draw (2.9,1) -- (4.2,1);
\coordinate (5) at (5.35,1) node at (5) {\small $\vep_{2m+r} {-} \vep_{m+r}$};
\draw (6.45,1) -- (7.65,1);
\coordinate (7) at (8.45,1) node at (7) {\small $\vep_{m+r} {-} \vep_r$};
\draw (8.45,.8) -- (4.6,.1);
\coordinate (0) at (4.1,0) node at (0) {\small $\vep_{r} {+} \vep_{\bar{r}}$};
\draw (3.7,-.15) -- (0,-.8);
\coordinate (2) at (0,-1) node at (2) {\small $\vep_{ m+\bar{r} } {-} \vep_{ \bar{r} }$};
\draw (.8,-1) -- (1.9,-1);
\coordinate (4) at (3,-1) node at (4) {\small $\vep_{ 2m+\bar{r} } {-} \vep_{ m+\bar{r} }$};
\draw (4.1,-1) -- (5.6,-1);
\coordinate (6) at (6,-1) node at (6) {\small \dots};
\draw (6.4,-1) -- (7.9,-1);
\coordinate (8) at (8.45,-1) node at (8) {\small $\alpha_{\bar{r} ,q(\bar{r})}$};
\end{tikzpicture} 
\end{center} 
which is of type $A_{d(r)}$ with $d(r) = q(r)+1 + q(\bar{r})$. 

\smallskip 
If $m =2k{+}1$ then $k < \bar{k}$ and we obtain the previous diagrams for $r = 1,\dots, k$. 
If $m=2k$ then $k = \bar{k}$, we obtain the previous diagrams only for $r = 1,\dots, k{-}1$, 
and the diagram $\mathcal{D}(k)$ consisting of 
the roots $\alpha_{k,b}$ for $b=1,\dots, q(k)$ has type $A_{q(k)}$.

\smallskip 
Moreover, there is a diagram $\calD(0)$ which is as follows: 
\begin{enumeratea}
\item In type $B_n$, $\calD(0)$ is the diagram below: 
\begin{center} 
\begin{tikzpicture}
\coordinate (0) at (.2,1) node at (0) {\small $\vep_m$};
\draw (.5,1.05) -- (1.8,1.05);
\draw (.5,.95) -- (1.8,.95);
\coordinate (2) at (1.15,1) node at (2) {$<$}; 
\coordinate (1) at (2.6,1) node at (1) {\small $\vep_{2m} {-} \vep_{m}$};
\draw (3.4,1) -- (5,1);
\coordinate (5) at (5.4,1) node at (5) {\small \dots};
\draw (5.7,1) -- (7.3,1);
\coordinate (7) at (8,1) node at (7) {\small $\alpha_{m,q(m)}$};
\end{tikzpicture} 
\end{center} 
which is of type $B_{1+q(m)}$. (If $q(m)=0$, \ie if $m> n/2$, it consists of $\vep_m$ only and has type $B_1$.) 

\smallskip 
\item In type $D_{n+1}$, one has $\bar{0} = m$, the roots $\alpha_{m,b}$ are among the roots $\alpha_{0,b}$, and one obtains the 
diagram $\mathcal{D}(0)$ below: 
\begin{center} 
\begin{tikzpicture}
\coordinate (1) at (.2,1) node at (1) {\small $\vep_{m} {-} \vep_0$};
\draw (.8,1) -- (2.2,1);
\coordinate (3) at (3,1) node at (3) {\small $\vep_{2m} {-} \vep_{m}$};
\draw (3.7,1) -- (5.3,1);
\coordinate (5) at (5.7,1) node at (5) {\small \dots};
\draw (6.1,1) -- (7.6,1);
\coordinate (7) at (8.2,1) node at (7) {\small $\alpha_{0,q(0)}$};
\draw (3,.8) -- (3,.2);
\coordinate (0) at (3,0) node at (0) {\small $\vep_{m} {+} \vep_0$};
\end{tikzpicture} 
\end{center} 
which is of type $D_{1+q(0)}$. (If $q(0) = 1$, \ie if $m >n/2$, it consists of $\vep_{m} {-} \vep_0$ and $\vep_{m} {+} \vep_0$ only
and has type $D_2$.)
\end{enumeratea}  

Next, for $r = 1,\dots, k-1$, set $\eta_r = -\vep_{\bar{r} + q(\bar{r}) m} - \vep_{r+1 + q(r+1)m}$. Then $\eta_r$ is orthogonal 
to all roots of $R_m$ except that it is linked by a simple bond to the right end-point of $\mathcal{D}(r)$, which is 
$\alpha_{\bar{r} ,q(\bar{r})}$ if $q(\bar{r})\geq 1$ and $\beta_r$ else, and to the left end-point of  
$\mathcal{D}(r+1)$, which is $\alpha_{r+1,q(r+1)}$  if $q(r+1)\geq 1$ and $\beta_{r+1}$ else.  
Moreover, the $\eta_r$'s are pairwise orthogonal. 

Further, set 
$$
\eta_0 = 
\begin{cases} 
-\vep_{m + q(m) m}  - \vep_{1+q(1)m} & \text{ in type } B_n; 
\\
-\vep_{q(0) m} - \vep_{1+q(1)m}& \text{ in type }  D_{n+1}. 
\end{cases}
$$
Then $\eta_0$ is orthogonal to all $\eta_r$'s for $r \geq 1$ and to all roots of $R_m$ except that it is linked by simple bonds to the 
left end-point of $\calD(1)$ and to the right end-point of $\calD(0)$, except if $\calD(0)$ is of type $B_1$ or $D_2$, in which case 
$\eta_0$ is linked by a double bond to $\vep_m$ or by simple bonds to both $\vep_{m} {-} \vep_0$ and $\vep_{m} {+} \vep_0$. 
Since we know that $\vert R_m \vert + k = n$ we obtain that the diagram below is of the same type as $R$, \ie $B_n$ or $D_{n+1}$: 
\begin{center} 
\begin{tikzpicture}
\coordinate (0) at (0,0) node at (0) {\small $\mathcal{D}(0)$};
\draw (.4,0) -- (1.8,0);
\coordinate (1) at (2,0) node at (1) {\small $\eta_0$};
\draw (2.2,0) -- (3.6,0);
\coordinate (2) at (4,0) node at (2) {\small \dots};
\draw (4.4,0) -- (5.6,0);
\coordinate (3) at (6,0) node at (3) {\small $\eta_{k-1}$};
\draw (6.4,0) -- (7.6,0);
\coordinate (4) at (8,0) node at (4) {\small $\mathcal{D}(k)$};
\end{tikzpicture} 
\end{center} 

This completes the description of the Dynkin diagram of $R_m$ and proves that  $R_m$ is of Levi type. Before proceeding further, 
let us indicate how the foregoing has to be modified when $n <m$, say $n = m -i$ with $i=1,\dots, \lfloor m/2 \rfloor$. 
In this case, $R_m$ consists of the roots $\vep_r + \vep_{m-r}$ for $r = i,\dots, k$ and $2r < m$, is of Levi type, and its type is: 
\begin{equation}\label{BDcase n < m}
\begin{cases}
(k-i) A_1 & \text{ if } m=2k; 
\\
(k-i+1) A_1 & \text{ if } m=2k+1. 
\end{cases}
\end{equation} 

Now, we come back to the case $m \leq n$. Combining the previous discussion with Lemma \ref{lem-BD}, one obtains the: 

\begin{cor}\label{cor-BD} Let $R$ be of type $B_n$ or $D_{n+1}$, with $n\geq 3$. 
The Dynkin diagram $\Gamma_m$ is the disjoint union of the diagrams $\calD(0), \dots, \calD(k)$ 
described above, except when $m=2k$ and $3k\leq n$. 

In this case, $\Gamma_m$ is the disjoint union of $\calD(0), \dots, \calD(k-1)$ and 
the following diagram $\mathcal{D}(k)^\dag$: 
\begin{center} 
\begin{tikzpicture}
\coordinate (0) at (-.5,0) node at (0) {\small $\alpha_{k,q(k)}$};
\draw (.1,0) -- (1.6,0);
\coordinate (1) at (2,0) node at (1) {\small \dots};
\draw (2.3,0) -- (3.8,0);
\coordinate (2) at (5,0) node at (2) {\small $\vep_{k+2m} {-} \vep_{k+m}$};
\draw (6.2,0) -- (7.8,0);
\coordinate (3) at (8.6,0) node at (3) {\small $\vep_{k+m} {-} \vep_{k}$};
\draw (5,-.2) -- (5,-1.2); 
\coordinate (4) at (5,-1.5) node at (4) {\small $\vep_{k+m} {+} \vep_{k}$};
\end{tikzpicture} 
\end{center} 
formed by $\calD(k)$ and the additional simple root $\vep_{m+k} + \vep_k$ of height $2m$, 
which has type $D_{1 + q(k)}$. Then $R(m)$ is not of Levi type and 
the constant $d_m$ of \cite{NPP25} is equal to $2^{q(k)}$. 
\end{cor} 

\begin{proof} The previous discussion completely describes $R_m$ so, by Lemma  \ref{lem-BD}, we may assume that 
$m = 2k$ and $3k \leq n$. Then 
the additional simple root $\delta_{2m} = \vep_{m+k} + \vep_k$ is orthogonal to all roots of $R_m$ except 
$\vep_{k+2m} {-} \vep_{k+m}$ and, together with $\calD(k)$ which has type $A_{q(k)}$, it forms the diagram 
$\calD(k)^\dag$ shown above, which has type $D_{1+q(k)}$. 
(It has type $D_2$ if $q(k) = 1$, \ie if $n < k +2m$.)

One knows in this case that $R(m)$ is not of Levi type by the classification up to 
$W$-conjugacy of closed root subsystems of $R$, see \eg \cite{Dyn52}. Finally, since $\delta_{2m}$ is one of 
the spin nodes of a diagram of type $D_{1+q(k)}$, the dimension $d_m$ of \cite{NPP25}
is the dimension of a spin representation in type $D_{1 + q(k)}$, hence $d_m = 2^{q(k)}$. 
\end{proof}

\smallskip 
Now, let us make explicit the $q(i)$'s and the types of $R_m$ and $\Gamma(m)$ when $m\leq n$. 
One has the following:

\begin{lem}\label{lem-BD-q} Write $n = qm+ t$, with $q\in \NN^*$ and $0\leq t < m$. Then: 
\begin{enumeratei}
\item $\calD(0)$ has type $B_q$, \resp $D_{q+1}$, if $R$ has type $B_n$, \resp $D_{n+1}$. 

\smallskip
\item Assume that $m=2k+1$. Then $\calD(1) \cup \cdots \calD(k)$ has type: 
$$
\begin{cases} 
i A_{2q+1} + (k-i) A_{2q} & \text{ if $t=k+i$ with } i=1,\dots, k; 
\\
t A_{2q} + (k-t) A_{2q-1} & \text{ if } t=0,\dots, k.
\end{cases}
$$

\smallskip 
\item Assume that $m=2k$. Then $\calD' = \calD(1) \cup \cdots \calD(k-1)$ and $\calD(k)$ have type: 
$$
\begin{array}{|c|c|c|}
\hline 
 & \calD' & \calD(k) 
\\
\hline 
\begin{array}{c}
t = k+i \\
i = 0,\dots, k{-}1 
\end{array} & (k{-1-}i) A_{2q} + i A_{2q+1} & A_q 
\\
\hline 
t = 0,\dots, k{-}1 & t A_{2q} + (k{-1-} t) A_{2q-1} & A_{q-1} \rule{0pt}{12pt} 
\\[2pt] 
\hline 
\end{array}
$$
\end{enumeratei}
\end{lem}  

\begin{proof} One has $n = qm+ t$ with $0\leq t< m$. Thus $q(0) = q$ and $q(m) = q-1$, and (i) follows. 
Now, recall that for $r = 1,\dots, \lfloor (m-1)/2\rfloor$ we set $\bar{r} = m-r$. 

Assume first that $m = 2k+1$. 
If $t \geq k+1$, say $t = k+i$ with $i = 1,\dots, k$, then all $q(r)$ equal $q$, whereas 
$q(\bar{r})= q$ for $k{-}i  < r \leq k$ and 
$q(\bar{r}) = q-1$ for $1\leq r < m - t = k +1-i$. This gives 
$i$ copies of $A_{2q+1}$ and $k-i$ copies of $A_{2q}$. 

\smallskip 
  If $t = 0,\dots, k$, then all $q(\bar{r})$ equal $q-1$, whereas $q(r) = q$ for $1\leq r \leq t$ and 
$q(r) = q{-}1$ for $t < r \leq k$. This gives $t$ copies of $A_{2q}$ and $k{-}t$ copies of $A_{2q-1}$.

\smallskip Assume now that $m = 2k$. 
If $t \geq k$, say $t = k+i$ with $i = 0,\dots, k-1$, then all $q(r)$ equal $q$, whereas 
$q(\bar{r})= q$ for $k{-}i  \leq r \leq k-1$ and 
$q(\bar{r}) = q-1$ for $1\leq r < m - t = k-i$. So $q(k) = q$ and 
we obtain that $\calD'$ is the sum of $i$ copies of $A_{2q+1}$ and $k{-1-}i$ copies of $A_{2q}$. 

\smallskip  If $t = 0,\dots, k-1$, then all $q(\bar{r})$ equal $q{-}1$, whereas $q(r) = q$ for $1\leq r \leq t$ and 
$q(r) = q-1$ for $t < r \leq k$. So $q(k) = q{-}1$ and we obtain that $\calD'$ is the sum of $t$ copies of $A_{2q}$ and $k{-1-}t$ copies of $A_{2q-1}$. 
 \end{proof} 

We summarize our findings in the following  proposition,
where we assume that $m\leq n$ since 
the case $n < m$ was treated in \eqref{BDcase n < m} above. 
Note that when $m=2k$, the condition $m\leq n < 3k$ corresponds to $n = m + t$ with $t = 0,\dots, k-1$.

\begin{prop}\label{prop-BD} Let $R$ be of type $B_n$ or $D_{n+1}$, with $n \geq 3$. Assume $m\leq n$, 
write $n = qm + t$ with $t = 0,\dots, m-1$, and set $k = \lfloor m/2\rfloor$. Denote by $BD_q$ a diagram of 
type $B_q$ if $R = B_n$ and of type $D_{q+1}$ if $R = D_{n+1}$. 
\begin{enumeratea} 
\item If $m = 2k+1$, then $R(m)$ is of Levi type and has type 
$$
BD_q + \begin{cases} 
i A_{2q+1} + (k-i) A_{2q} & \text{ if $t=k+i$ with } i=1,\dots, k; 
\\
t A_{2q} + (k-t) A_{2q-1} & \text{ if } t=0,\dots, k.
\end{cases}
$$

\item If $m = 2k$ and $n = m + t$ with $t = 0,\dots, k-1$ then $R(m)$ has type 
$BD_1 + t A_{2} + (k{-1-}t) A_{1}$ and is of Levi type. 

\smallskip 
\item If $m = 2k$ and $n\geq 3k$, then $R(m)$ is not of Levi type. 
With the notation of\, {\rm \Th \ref{main-thm}},  one has the following table: 
$$
\begin{array}{|c|c|c|c|} 
\hline 
 & \Gamma(m) & X^0 \rule{0pt}{11pt} & d_m 
\\
\hline 
\begin{array}{c}
t = k+i \\
i = 0,\dots, k{-}1 
\end{array}  & BD_q +(k{-1-}i) A_{2q} + i A_{2q+1}  + D_{q+1}^\dag & A_q & 2^q 
\\
\hline 
\begin{array}{c}
q \geq 2 \\
t = 0,\dots, k{-}1 
\end{array} & BD_q + t A_{2q} + (k{-1-} t) A_{2q-1} + D_{q}^\dag & A_{q-1} & 2^{q-1} 
\\ 
\hline 
\end{array} 
$$
\end{enumeratea}  
\end{prop}

\subsection{Types $E_6, E_7, E_8$}\label{subsec-E} 
We number the nodes of the Dynkin diagram as in \cite{Bou68}, Tables 
V--VII, \ie the nodes $1,2,4,\dots, 8$ are linearly ordered and the nodes $4$ and $2$ are connected. 
Using \S \ref{dual-part} we obtain the following tables for the $\pi_k$'s. 
$$
\begin{array}{|c|c|}  
\hline
\text{\small height} & E_6 
\\
\hline 
2-4 &  5
\\
\hline 
5 &   4 
\\
\hline
6-7 &   3
\\
\hline
8 &  2 
\\
\hline
9-11 &   1 
\\
\hline
\end{array} 
\qquad \qquad 
\begin{array}{|c|c|c|}
\hline 
\text{\small height}  & E_7 & E_8 
\\
\hline 
2-5 & 6 & 7 
\\
\hline 
6-7 & 5 & 7 
\\
\hline 
8-9 & 4 & 6
\\
\hline 
10-11 & 3 & 6
\\
\hline 
12-13 & 2 & 5
\\
\hline 
14-17 & 1 & 4
\\
\hline 
\end{array} 
\qquad \qquad 
\begin{array}{|c|c|}
\hline 
\text{\small height} & E_8 
\\
\hline 
18-19 & 3
\\
\hline 
20-23 & 2
\\
\hline 
24-29 & 1
\\
\hline 
\end{array}  
$$

For any connected subset $I = \{i,j,k,\dots\}$ of the Dynkin diagram, we denote by $\alpha_{i,j,k,\dots}$ the root with support 
$I$ whose coefficients are all $1$. We consider the roots $\gamma_4 = \alpha_{2,3,4,5}$ and 
$\gamma_8 = \alpha_{[2,8]} + \alpha_4$ (which correspond to roots in $D_4$ and $D_7$ respectively) and:  
$$
\delta_6 = \alpha_{[1,6]} \qquad\quad  \delta_{10} = \EEE8{1}{2}{2}{1}{1}{1}{1}{1} \qquad \quad \delta_{16} = 
\EEE8{1}{2}{3}{3}{3}{2}{1}{1} 
$$
Each of these root is of height $2m$, for some $m$, and not the sum of two roots of height $m$.

\smallskip 
With notation as in \Th \ref{main-thm}, and using the notation $[4A_1]''$ of \cite{Dyn52},  \S II.5.17, Table 11, 
 one has the: 

\begin{prop} Let $R$ be of type $E_6,E_7, E_8$ and let $h$ be its Coxeter number.

\smallskip 
{\rm (a)}  If $m\geq h/2$ then $R(m)$ 
has type $\eta_m A_1$ and is of Levi type. 

\smallskip 
{\rm (b)}  Else, $R(m)$, $X^\dag$ and $d_m$ are described by the following tables for $E_6, E_7$: 
$$
\begin{array}{|c|c|c|c|c|}
\hline 
 & \multicolumn{2}{c|}{E_6,\;  h=12} &  \multicolumn{2}{c|}{E_7, \; h=18}
 \\
 \hline 
m, \delta_{2m}  & \Gamma(m) & X^0, d_m & \Gamma(m) & X^0, d_m 
\\
\hline 
2 , \gamma_4  & A_5 + A_1^\dag & \varnothing, 2 & A_7^\dag &  A_6, 8 \rule{0pt}{12pt} 
\\
\hline
3 , \delta_6 & 2A_2 + A_2^\dag & A_1, 3 & A_5 + A_2^\dag &  A_1, 3 \rule{0pt}{12pt} 
\\
\hline 
4  & 2A_2 + A_1 & & A_4 + A_2 &  
\\ 
\hline 
5 & A_2 + 2A_1 &  & A_3 + A_2 + A_1 &  
\\
\hline 
6 & 3A_1   & & 2 A_2 + A_1 & 
\\
\hline 
7 & 3A_1  & & A_2 + 3 A_1 &   
\\
\hline 
8 & 2 A_1 &   & A_2 + 2 A_1 &  
\\
\hline 
9  &  & &   [4 A_1]'' &  \rule{0pt}{11pt}
\\
\hline 
\end{array}
$$
and for $E_8$, where $h=30$: 
$$
\begin{array}{|c|c|c|}
\hline 
m, \delta_{2m}  & \Gamma(m) & X^0, d_m 
\\
\hline 
2 , \gamma_4   &  D_8^\dag & D_7, 16 \rule{0pt}{12pt} 
\\
\hline
3 , \delta_6 &  A_8^\dag & A_7, 9  \rule{0pt}{12pt} 
\\
\hline 
4  & D_5 + A_3^\dag & A_2, 4 \rule{0pt}{12pt} 
\\
\hline 
5  &  A_4 + A_4^\dag & A_3, 5 \rule{0pt}{12pt} 
\\
\hline 
6 &   A_4 + A_3 & 
\\
\hline 
7 &   A_4 + A_2 + A_1 & 
\\
\hline 
8  &   A_3 + A_2 + A_1 + A_1^\dag & \varnothing, 2 \rule{0pt}{12pt} 
\\
\hline 
\end{array}
\quad 
\begin{array}{|c|c|}
\hline 
m & \Gamma(m) 
\\
\hline 
9 & A_3 + A_2 + A_1 
\\
\hline
10 & 2A_2 + 2A_1 
\\
\hline 
11 &2A_2 + 2A_1 
\\
\hline 
12 & A_2 + 3A_1  
\\
\hline 
13 &    A_2 + 3A_1  
\\
\hline 
14 &    A_2 + 2A_1   
\\
\hline 
\end{array}
$$

\noindent In these cases, $R(m)$ is not of Levi type exactly when there is a component $X^\dag$, that is, when $m=2,3$ or 
$R$ is of type $E_8$ and $m = 4,5,8$. 
\end{prop} 

\begin{proof} For each value of $m < h/2$, we exhibit the Dynkin diagram of $R_m$, together with the additional root $\delta_{2m}$ if any, and check that the set of positive roots so obtained has cardinality $\vert R(m)\vert$. From the previous tables, one sees that these cardinalities are given by the following tables: 

$$
\begin{array}{|c|c|c|c|} 
\hline 
m & E_6 & E_7 & E_8 
\\
\hline 
2 & 16 & 28 & 56 
\\
\hline 
3 & 9 & 18 & 36 
\\
\hline 
4 & 7 & 13 & 26 
\\
\hline 
5 & 5 & 10 & 20 
\\
\hline 
\end{array}
\qquad \quad 
\begin{array}{|c|c|c|}
\hline 
m &  E_7 & E_8 
\\
\hline 
6 &   7 & 16 
\\
\hline 
7 &  6 & 14 
\\
\hline 
8 &  5 & 11 
\\
\hline
9 & 4 & 10 
\\
\hline 
\end{array}
\qquad \quad 
\begin{array}{|c|c|}
\hline 
m & E_8 
\\
\hline 
10 & 8 
\\
\hline 
11 & 8
\\
\hline 
12 & 6
\\
\hline 
13 & 6
\\
\hline 
14 & 5 
\\
\hline 
\end{array}
$$

Suppose that $m=2$. Clearly, the root $\gamma_4$ of height $4$ cannot be written as a sum of two elements of $R_2$. 
For $E_6$, the elements of $R_2$ form the Dynkin diagram: 
\begin{center} 
\begin{tikzpicture}
\coordinate (1) at (-4,0) node at (1) {$\alpha_{4,5}$};
\draw (-3.6,.05) -- (-2.4,.05);
\coordinate (2) at (-2,0) node at (2) {$\alpha_{1,3}$};
\draw (-1.6,.05) -- (-.4,.05);
\coordinate (3) at (0,0) node at (3) {$\alpha_{4,2}$};
\draw (.4,.05) -- (1.6,.05);
\coordinate (4) at (2,0) node at (4) {$\alpha_{5,6}$};
\draw (2.4,.05) -- (3.6,.05);
\coordinate (5) at (4,0) node at (5) {$\alpha_{3,4}$};
\end{tikzpicture} 
\end{center} 
\noindent 
and $\gamma_4$ is orthogonal to them. Thus $R_2 \cup \{\gamma_4\}$ is a base of a root system $A_5 + A_1^\dag$ which 
gives  $16 = \vert R(2) \vert$ positive roots. 

For $E_7$, the additional element $\alpha_{6,7}$ of $R_2$ is connected to $\alpha_{4,5}$ and also to $\gamma_4$, so we 
obtain a root sytem $A_7^\dag$, which contains $28 = \vert R(2) \vert$ positive roots. 
Next, for $E_8$ the additional element $\alpha_{7,8}$ of $R_2$ is connected to $\alpha_{5,6}$; together with $\gamma_4$, this 
gives a root sytem $D_8^\dag$, whence $56 = \vert R(2) \vert$ positive roots. 

\smallskip Suppose that $m=3$. Clearly, the root $\delta_6$ of height $6$ cannot be written as a sum of two elements of $R_3$. 
For $E_6$, the elements of $R_3$ form the  Dynkin diagrams: 
\begin{center}   
\begin{tikzpicture}
\coordinate (1) at (-4,0) node at (1) {$\alpha_{4,5,6}$};
\draw (-3.55,.05) -- (-2.45,.05);
\coordinate (2) at (-2,0) node at (2) {$\alpha_{3,4,2}$};
\coordinate (3) at (0,0) node at (3) {$\alpha_{1,3,4}$};
\draw (.45,.05) -- (1.55,.05);
\coordinate (4) at (2,0) node at (4) {$\alpha_{2,4,5}$};
\coordinate (5) at (4,0) node at (5) {$\alpha_{3,4,5}$};
\end{tikzpicture} 
\end{center} 
and $\delta_6$ is connected to $\alpha_{3,4,5}$. Thus one obtains $2A_2 + A_2^\dag$, which gives $9  = \vert R(3)\vert$ positive roots. 

For $E_7$, the additional element $\alpha_{5,6,7}$ of $R_3$ connects $\alpha_{3,4,2}$ and $\alpha_{1,3,4}$, thus one 
obtains $A_5 + A_2^\dag$, which gives $18  = \vert R(3)\vert$ positive roots. Next, for $E_8$ the additional element 
$\alpha_{6,7,8}$ of $R_3$ connects $\alpha_{2,4,5}$ and $\alpha_{3,4,5}$, thus one 
obtains $A_8^\dag$, which gives $36  = \vert R(3)\vert$ positive roots. 

\smallskip Suppose that $m=4$. For $E_6$, the elements of $R_4$ form the Dynkin diagrams: 
\begin{center}   
\begin{tikzpicture}
\coordinate (1) at (-4,0) node at (1) {$\alpha_{3,4,5,6}$};
\draw (-3.45,.05) -- (-2.55,.05);
\coordinate (2) at (-2,0) node at (2) {$\alpha_{1,3,4,2}$};
\coordinate (3) at (0,0) node at (3) {$\alpha_{2,3,4,5}$};
\coordinate (4) at (2,0) node at (4) {$\alpha_{2,4,5,6}$};
\draw (2.55,.05) -- (3.45,.05);
\coordinate (5) at (4,0) node at (5) {$\alpha_{1,3,4,5}$};
\end{tikzpicture} 
\end{center} 
This gives $2A_2 + A_1$, whence $7 = \vert R(4)\vert$ positive roots. For $E_7$, the root $\alpha_{4,5,6,7}$ connects 
$\alpha_{1,3,4,2}$ and $\alpha_{2,3,4,5}$ and this gives $A_4 + A_2$, whence $13 = \vert R(4)\vert$ positive roots.
Finally, for $E_8$  the root $\alpha_{4,5,6,7}$ is connected to $\alpha_{1,3,4,2}$ and the root $\gamma_8$ is connected 
to $\alpha_{1,3,4,5}$. Thus one obtains $D_5 + A_3^\dag$, which gives $26 = \vert R(4)\vert$ positive roots. 

\smallskip Suppose that $m=5$. For $E_6$, the elements of $R_5$ form the Dynkin diagrams: 
\begin{center}   
\begin{tikzpicture}
\coordinate (1) at (-5,0) node at (1) {$\alpha_{1,3,4,5,6}$};
\draw (-4.35,.05) -- (-3.1,.05);
\coordinate (2) at (-2.5,0) node at (2) {$\gamma_4 + \alpha_4$};
\coordinate (3) at (0,0) node at (3) {$\alpha_{[2,6]}$};
\coordinate (4) at (2.5,0) node at (4) {$\alpha_{[1,5]}$};
\end{tikzpicture} 
\end{center} 
This gives $A_2 + 2A_1$, whence $5 = \vert R(5)\vert$ positive roots. For $E_7$, the root $\alpha_{2,4,5,6,7}$ is connected 
to $\alpha_{1,3,4,5,6}$ whereas $\alpha_{[3,7]}$ is connected to $\alpha_{[1,5]}$. This gives $A_3 + A_2 + 2A_1$, whence 
$10 = \vert R(5)\vert$ positive roots. Finally, for $E_8$  the root $\alpha_{[4,8]}$ connects 
$\alpha_{[2,6]}$ and $\alpha_{[1,5]}$, whereas  $\delta_{10}$ is connected 
to $\alpha_{2,4,5,6,7}$. Thus one obtains $A_4 + A_4^\dag$, which gives $20 = \vert R(5)\vert$ positive roots. 
This completes the discussion for $E_6$. 

\smallskip Suppose that $m=6$. For $E_7$, the elements of $R_6$ form the Dynkin diagrams: 
\begin{center}   
\begin{tikzpicture}
\coordinate (1) at (-5,0) node at (1) {$\delta_6$};
\coordinate (2) at (-2.5,0) node at (2) {$\alpha_{[1,5]}+\alpha_4$};
\draw (-1.7,.05) -- (-.4,.05);
\coordinate (3) at (0,0) node at (3) {$\alpha_{[2,7]}$};
\coordinate (4) at (2.5,0) node at (4) {$\alpha_{[2,6]}+\alpha_4$};
\draw (3.3,.05) -- (4.4,.05);
\coordinate (5) at (5,0) node at (5) {$\alpha_{1\cup [3,7]}$};
\end{tikzpicture} 
\end{center} 
For $E_8$  the root $\alpha_{[3,8]}$ connects $\delta_6$ and $\alpha_{[1,5]}+\alpha_4$, 
whereas  $\alpha_{2\cup [4,8]}$ is connected to $\alpha_{1\cup [3,7]}$. 
 Thus one obtains $A_4 + A_3$, which gives $16 = \vert R(6)\vert$ positive roots. 
 
\smallskip Suppose that $m=7$. Note that $\alpha,\beta \in R_7$ are connected if and only if $\alpha+\beta$ is a root, so by looking at the list of roots of height $14$ (\cite{Bou68}, Tables VI--VII), we obtain 
that for $E_8$ the elements of $R_7$ form the following Dynkin diagrams: 
\begin{center}   
\begin{tikzpicture}
\coordinate (1) at (-6,1.2) node at (1) {$\EEE8{1}{2}{2}{1}{0}{0}{0}{1}$};
\coordinate (3) at (-3,1.2) node at (3) {$\EEE8{1}{1}{2}{1}{1}{0}{0}{1}$};
\draw (-2.15,1.3) -- (-.85,1.3);
\coordinate (5) at (0,1.2) node at (5) {$\EEE8{0}{1}{1}{1}{1}{1}{1}{1}$};
\coordinate (2) at (-6,0) node at (2) {$\EEE8{0}{1}{2}{1}{1}{1}{0}{1}$};
\draw (-5.15,.1) -- (-3.85,.1);
\coordinate (4) at (-3,0) node at (4) {$\EEE8{0}{1}{1}{1}{1}{1}{1}{0}$};
\draw (-2.15,.1) -- (-.85,.1);
\coordinate (6) at (0,0) node at (6) {$\EEE8{0}{1}{2}{2}{1}{0}{0}{1}$};
\draw (.85,.1) -- (2.15,.1);
\coordinate (8) at (3,0) node at (8) {$\EEE8{1}{1}{1}{1}{1}{1}{0}{1}$};
\end{tikzpicture} 
\end{center} 
This is $A_4 + A_2 + A_1$, whence $14 = \vert R(7)\vert$ positive roots. And the five roots belonging to $E_7$ give 
$A_2 + 3A_1$.  

\smallskip Suppose that $m=8$. For $E_7$, the elements of $R_6$ form the Dynkin diagrams:  
\begin{center}   
\begin{tikzpicture}
\coordinate (2) at (-6,0) node at (2) {$\EEE8{0}{1}{2}{2}{1}{1}{0}{1}$};
\draw (-5.15,.1) -- (-3.85,.1);
\coordinate (4) at (-3,0) node at (4) {$\EEE8{1}{2}{2}{1}{1}{0}{0}{1}$};
\coordinate (6) at (0,0) node at (6) {$\EEE8{1}{1}{2}{2}{1}{0}{0}{1}$};
\coordinate (8) at (3,0) node at (8) {$\EEE8{1}{1}{2}{1}{1}{1}{0}{1}$};
\end{tikzpicture} 
\end{center} 
This is $A_2 + 2A_1$, whence $5 = \vert R(8)\vert$ positive roots. For $E_8$, the elements $\alpha_{[1,8]}$ and 
$\alpha_{[2,8]}+\alpha_4$ are connected  to the first, \resp third, root in the above diagram, and the additional root 
$\delta_{16}$ is orthogonal to all of them. This gives $A_3 + A_2 + A_1 + A_1^\dag$, whence $11 = \vert R(8)\vert$
positive roots.

\smallskip Suppose that $m=9$ and 
$R$ has type $E_8$. Again, $\alpha,\beta \in R_7$ are connected if and only if $\alpha+\beta$ is a root, so by looking at the list of roots of height $18$ in \cite{Bou68}, Table VII, we obtain 
that the elements of $R_9$ form the following Dynkin diagrams: 
\begin{center}   
\begin{tikzpicture}
\coordinate (1) at (-6,1.2) node at (1) {$\EEE8{1}{1}{2}{2}{1}{1}{0}{1}$};
\coordinate (3) at (-3,1.2) node at (3) {$\EEE8{1}{2}{2}{1}{1}{1}{0}{1}$};
\draw (-2.15,1.3) -- (-.85,1.3);
\coordinate (5) at (0,1.2) node at (5) {$\EEE8{0}{1}{2}{2}{1}{1}{1}{1}$};
\coordinate (2) at (-6,0) node at (2) {$\EEE8{1}{2}{2}{2}{1}{0}{0}{1}$};
\draw (-5.15,.1) -- (-3.85,.1);
\coordinate (4) at (-3,0) node at (4) {$\EEE8{1}{1}{2}{1}{1}{1}{1}{1}$};
\draw (-2.15,.1) -- (-.85,.1);
\coordinate (6) at (0,0) node at (6) {$\EEE8{0}{1}{2}{2}{2}{1}{0}{1}$};
\end{tikzpicture} 
\end{center} 
This is $A_3 + A_2 + A_1$, whence $10 = \vert R(9)\vert$ positive roots, thus $R(9)$ has type  $A_3 + A_2 + A_1$ 
for $E_8$. 
Further, for $E_7$ the highest root is the fundamental weight $\omega_1$ and 
 the four elements of height $9$ which belong to $E_7$ fit into the following Dynkin diagram of type $E_7$: 
\begin{center}   
\begin{tikzpicture}
\coordinate(1) at (-.6,0) node at (1) {\small $-\alpha_{3,4}$}; 
\draw (-.1,0) -- (.8,0);
\coordinate(3) at (1.6,0) node at (3) {$\EE7{1}{2}{2}{1}{1}{1}{1}$};
\draw (2.4,0) -- (3.3,0);
\coordinate(4) at (3.7,0) node at (4) {\small $-\omega_1$};
\draw (4.1,0) -- (5.1,0);
\coordinate (5) at (5.9,0) node at (5) {$\EE7{1}{1}{2}{2}{1}{1}{1}$};
\draw (6.7,0) -- (7.6,0);
\coordinate(6) at (8.2,0) node at (6) {\small $-\alpha_{[5,7]}$}; 
\draw (8.8,0) -- (9.7,0);
\coordinate(7) at (10.5,0) node at (7) {$\EE7{0}{1}{2}{2}{2}{1}{1}$};
\draw (3.7,-.2) -- (3.7,-1);
\coordinate(2) at (3.8,-1.3) node at (2)  {$\EE7{1}{2}{2}{2}{1}{0}{1}$};
\end{tikzpicture} 
\end{center} 
This shows that $R(9)$ for $E_7$ is of Levi type: it is $W$-conjugate to the subsystem formed by the simple roots 
$\alpha_3, \alpha_2, \alpha_5, \alpha_7$, which is denoted $[4 A_1]''$ by Dynkin, see \cite{Dyn52},  \S II.5.17, Table 11. 
This completes the discussion for $E_7$, so from now on we assume that $R$ is of type $E_8$. 

\smallskip For $m = 10, 11$ (\resp $m=12,13,14$), since $\vert R(m)\vert - \vert R_m \vert$ equals $2$ (\resp $1$), 
it suffices to exhibit two pairs (\resp one pair) of elements of $R_m$ whose sum is a root. We indicate these pairs 
in the following table:  
\begin{center}   
\begin{tikzpicture}
\coordinate (0) at (-8,2.2) node at (0) {$m=10$}; 
\coordinate (2) at (-5.5,2.2) node at (2) {$\EEE8{0}{1}{2}{2}{2}{1}{1}{1}$};
\draw (-4.65,2.3) -- (-3.85,2.3);
\coordinate (4) at (-3,2.2) node at (4) {$\EEE8{1}{2}{2}{2}{1}{1}{0}{1}$};
\coordinate (6) at (0,2.2) node at (6) {$\EEE8{1}{1}{2}{2}{2}{1}{0}{1}$};
\draw (.85,2.3) -- (1.65,2.3);
\coordinate (8) at (2.5,2.2) node at (8) {$\EEE8{1}{2}{2}{1}{1}{1}{1}{1}$};
\coordinate (1) at (-8,1.1) node at (1) {$m=11$}; 
\coordinate (3) at (-5.5,1.1) node at (3) {$\EEE8{0}{1}{2}{2}{2}{2}{1}{1}$};
\draw (-4.65,1.2) -- (-3.85,1.2);
\coordinate (5) at (-3,1.1) node at (5) {$\EEE8{1}{2}{3}{2}{1}{0}{0}{2}$};
\coordinate (7) at (0,1.1) node at (7) {$\EEE8{1}{2}{3}{2}{1}{1}{0}{1}$};
\draw (.85,1.2) -- (1.65,1.2);
\coordinate (9) at (2.5,1.1) node at (9) {$\EEE8{1}{1}{2}{2}{2}{1}{1}{1}$};
\end{tikzpicture} 
\end{center} 
\begin{center}   
\begin{tikzpicture}
\coordinate (10) at (-8,0) node at (10) {$m=12$}; 
\coordinate (12) at (-5.5,0) node at (12) {$\EEE8{1}{2}{2}{2}{2}{1}{1}{1}$};
\draw (-4.65,.1) -- (-3.85,.1);
\coordinate (14) at (-3,0) node at (14) {$\EEE8{1}{2}{3}{2}{1}{1}{0}{2}$};
\coordinate (11) at (-8,-1.1) node at (11) {$m=13$}; 
\coordinate (13) at (-5.5,-1.1) node at (13) {$\EEE8{1}{2}{3}{2}{1}{1}{1}{2}$};
\draw (-4.65,-1) -- (-3.85,-1);
\coordinate (15) at (-3,-1.1) node at (15) {$\EEE8{1}{2}{3}{3}{2}{1}{0}{1}$};
\coordinate (16) at (-8,-2.2) node at (16) {$m=14$}; 
\coordinate (18) at (-5.5,-2.2) node at (18) {$\EEE8{1}{2}{3}{2}{2}{2}{1}{1}$};
\draw (-4.65,-2.1) -- (-3.85,-2.1);
\coordinate (20) at (-3,-2.2) node at (20) {$\EEE8{1}{2}{3}{3}{2}{1}{0}{2}$};
\end{tikzpicture} 
\end{center} 
This completes the proof of assertion (b). Then the assertions about Levi type follow from Dynkin 
classification, see \cite{Dyn52},  \S II.5.17, Table 11. 
\end{proof}

\subsection{Types $F_4$ and  $G_2$} In these types, we will denote by $A_k(s)$ a root system of type $A_k$ consisting of short roots. 

\smallskip For $G_2$, let the simple roots be $\beta$ (long) and $\alpha$ (short). 
The only non trivial case is $m=2$. Then $R^+(2) =\{\beta+\alpha, \beta+3\alpha\}$. 
Thus $R(2)$ is of type $A_1(s) + A_1^\dag$, not of Levi type, and one has $d_2 = 2$. 

\smallskip In type $F_4$, let the simple roots be $a, b$ (long) and $c, d$ (short), with  $b+c\in R$. 
The positive roots are given by the following table: 

$$
\begin{array}{|c|c|c|} 
\hline
\text{height} & \text{long roots} & \text{short roots}  
\\
\hline 
2 & a+b & b+c \qquad c+d  
\\
\hline
3 & b+2c & a+b+c \qquad b+c+d   
\\
\hline
4 & a+b+2c & a+b+c+d \qquad b+2c+d  
\\
\hline
5 & a+2b+2c \qquad b+2c+2d &  a+b+2c+d   
\\
\hline
6 & a+b+2c+2d  & a+2b+2c+d  
\\
\hline
7 & a+2b+2c+2d & a+2b+3c+d   
\\
\hline
8 & & a+2b+3c+2d   
\\
\hline
9 & a+2b+4c+2d &   
\\
\hline
10 & a+3b+4c+2d &  
\\
\hline
11 &  2a+3b+4c+2d &  
\\
\hline
\end{array} 
$$

\smallskip 
One sees easily that $R(6), R(7)$ are of type $A_1 + A_1(s)$ and that 
$R(5)$, \resp $R(4)$, is of type $A_2 + A_1(s)$, \resp $A_2(s) + A_1$. By Dynkin's classification, all are of Levi type. 

For $m=2,3$, the long roots $\delta_4 = a+b+2c$ and $\delta_6 = a+b+ 2c+2d$ are not the sum of two elements of $R_m$. 

For $m=2$, since $\delta_4$ is orthogonal to $R_2$, which is of type $C_3$, one obtains that $R(2)$ is of type 
$C_3 + A_1^\dag$, not of Levi type, and $d_2 = 2$. 

For $m=3$, the elements of $R_3$ form the following Dynkin diagrams:
\begin{center}   
\begin{tikzpicture}
\coordinate (2) at (-3,0) node at (2) {$b+2c$};
\coordinate (3) at (0,0) node at (3) {$a+b+c$};
\draw (.75,0) -- (2.25,0);
\coordinate (4) at (3,0) node at (4) {$b+c+d$};
\end{tikzpicture} 
\end{center} 
and $\delta_6$ is connected to $b+2c$ only. Thus $R(3)$ is of type $A_2(s) + A_2^\dag$, not of Levi type, and $d_3 = 3$. 
This completes the proof of Theorem \ref{main-thm}.


\begin{thebibliography}{NPP25} 

\bibitem[Bou68]{Bou68} N. Bourbaki, Groupes et alg\`ebres de Lie, Ch.~4-6, Hermann, 1968. English transl.: 
Lie Groups and Lie Algebras, Ch.~4-6, Springer-Verlag, 2002.   

\bibitem[Dyn52]{Dyn52} E. B. Dynkin, Semisimple subalgebras of semisimple Lie algebras, pp.175-308 \emph{in}: Selected Papers of E. B. Dynkin, Amer. Math. Soc.. \& International Press Boston, 2000 [transl. of Russian original: 
Mat. Sb. (N.S.) {\bf 30} (1952), 349-462]. 


\bibitem[Kos59]{Kos59} B. Kostant, The principal three-dimensional subgroup and the Betti numbers of a complex simple Lie group, Amer. J. Math. {\bf 81} (1959), 073-1032. 

\bibitem[Mac72]{Mac72} I. G. Macdonald, The Poincar\'e Series of a Coxeter Group, Math. Ann. {\bf 199} (1972), 161-174. 

\bibitem[NPP25]{NPP25} S. Nadimpalli, S. Pattanayak, D. Prasad, Character Theory at a Torsion Element, arXiv:2504.14684. 
\end{thebibliography}
\end{document}